\documentclass[11pt]{amsart}
\usepackage{epsfig}
\usepackage{amsbsy,latexsym}
\usepackage{amsmath}
\usepackage{amssymb, mathrsfs}
\usepackage[mathscr]{eucal}
\usepackage{hyperref}
\usepackage{amsfonts}
\usepackage{amsthm}
\usepackage{amsmath}
\usepackage{amsfonts}
\usepackage{latexsym}
\usepackage{amssymb}
\usepackage{amscd}
\usepackage[latin1]{inputenc}
\usepackage{verbatim}

\usepackage[english]{babel}

\newtheorem{definition}{Definition}
\newtheorem{proposition}[definition]{Proposition}
\newtheorem{lemma}[definition]{Lemma}

\newtheorem{theorem}[definition]{Theorem}
\newtheorem{corollary}[definition]{Corollary}
\newtheorem{conjecture}[definition]{Conjecture}

\newtheorem{remark}[definition]{Remark}
\newtheorem{example}[definition]{Example}
\newtheorem{question}[definition]{Question}

\def\bcj{\begin{conjecture}}
\def\ecj{\end{conjecture}}
\def\bcr{\begin{corollary}}
\def\ecr{\end{corollary}}
\def\bd{\begin{definition}}
\def\ed{\end{definition}}
\def\bea{\begin{eqnarray}}
\def\eea{\end{eqnarray}}
\def\bem{\begin{enumerate}}
\def\eem{\end{enumerate}}
\def\bex{\begin{example}}
\def\eex{\end{example}}
\def\bim{\begin{itemize}}
\def\eim{\end{itemize}}
\def\bl{\begin{lemma}}
\def\el{\end{lemma}}
\def\bpf{\begin{proof}}
\def\epf{\end{proof}}
\def\bpp{\begin{proposition}}
\def\epp{\end{proposition}}
\def\bqu{\begin{question}}
\def\equ{\end{question}}
\def\br{\begin{remark}}
\def\er{\end{remark}}
\def\bt{\begin{theorem}}
\def\et{\end{theorem}}

\def\btb{\begin{tabular}}
\def\etb{\end{tabular}}

\newcommand{\ket}[1]{|#1\rangle}

\newcommand{\norm}[1]{\lVert#1\rVert}
\newcommand{\abs}[1]{|#1|}
\newcommand{\Tr}{\mathrm{Tr}}

\begin{document}

\title[Isoclinic Subspaces and Quantum Error Correction]{Isoclinic Subspaces and Quantum Error Correction}

\author[D.~W. Kribs, D. Mammarella, R. Pereira]{David~W.~Kribs$^{1,2}$, David Mammarella$^{1}$, Rajesh Pereira$^{1}$}

\address{$^1$Department of Mathematics \& Statistics, University of Guelph, Guelph, ON, Canada N1G 2W1}
\address{$^2$Institute for Quantum Computing, University of Waterloo, Waterloo, ON, Canada N2L 3G1}




\begin{abstract}
We exhibit equivalent conditions for subspaces of an inner product space to be isoclinic, including a characterization based on the classical notion of canonical angles. We identify a connection with quantum error correction, showing that every quantum error correcting code is associated with a family of isoclinic subspaces, and we prove a converse for pairs of such subspaces. We also show how the canonical angles for isoclinic subspaces arise in the structure of the higher rank numerical ranges of the corresponding orthogonal projections. 
\end{abstract}

\maketitle

\section{Introduction}

The classical notions of canonical angles and isoclinic subspaces have played a role in Euclidean geometry, and in matrix and operator theory and beyond for over a century \cite{jordan1875essai,balla2019equiangular, bjorck1973numerical, hoggar1977new, wong1960clifford, wong1977linear}. On the other hand, quantum information theory is relatively new, with roots going back several decades but only emerging as a formal field of study over the past quarter century or so \cite{nielsen2002quantum}. Quantum error correction is a fundamental subfield with aspects touching on all parts of quantum information, from theory to experiment \cite{shor1995pw,steane1996error,gottesman1996d,bennett1996ch,knill2000theory,kribs2005quantum}.

In this paper, we bring together equivalent conditions for isoclinic subspaces, including a new description based on canonical angles. We establish connections with the theory of quantum error correction, showing how quantum error correcting codes are associated with families of isoclinic subspaces. We also show how higher rank numerical ranges of matrices, originally introduced for quantum error correction purposes \cite{choi2005quantum,choi2006higher2,li2007higher,woerdeman2008higher,martinez2008higher,choi2008geometry,li2008canonical,li2009condition,li2011generalized,GLPS}, arise in the study of isoclinic subspaces.

The paper is organized as follows. The next section includes a review of the classical notions of canonical angles and isoclinic subspaces, and we give equivalent conditions for families of subspaces to be isoclinic. In the following section we show how every quantum error correcting code and error model determines a family of isoclinic subspaces and we prove a converse for pairs of such subspaces. In the final section we show how the canonical angles for isoclinic subspaces are embedded in the structure of the higher rank numerical ranges for the corresponding orthogonal projections. We also include a pair of illustrative examples.

\section{Canonical Angles and Isoclinic Subspaces}

We first introduce the classical notion of canonical angles between pairs of subspaces. These are sometimes referred to as principal angles and were first formulated by Jordan \cite{jordan1875essai}.

\begin{definition}
{\rm
Let $\mathcal V$ and $\mathcal W$ be finite dimensional subspaces of a Hilbert space $\mathcal{H}$ and let $l = \min\{\dim(\mathcal V), \dim(\mathcal W)\}.$ Then the {\it canonical angles} $\{ \theta_1, \ldots , \theta_l  \}$ between $\mathcal V$ and $\mathcal W$ are defined as follows: the first canonical angle is the unique number $\theta_1 \in [0,\frac{\pi}{2}]$ such that
\[
\cos(\theta_1) = \max \{\abs{\langle x, y \rangle} : x \in\mathcal V, y\in\mathcal W, \norm{x} = \norm{y} = 1\}.
\]
Let $x_1$ and $y_1$ be unit vectors in $\mathcal V$ and $\mathcal W$ for which the previous maximum is attained. Then we define the second canonical angle as the unique number $\theta_2 \in [0,\frac{\pi}{2}]$ such that
\[
\cos(\theta_2) = \max \{\abs{\langle x, y \rangle} : x \in\mathcal V\cap \{x_1\}^\perp, y \in\mathcal W \cap \{ y_1\}^\perp, \norm{x} = \norm{y} = 1 \}.
\]
For each $k \leq l$, similarly choose unit vectors $x_2, \ldots x_{k-1}$ and $y_2, \ldots y_{k-1}$ in $\mathcal V$ and $\mathcal W$ respectively, in each case where the previous maximum is attained. Then  $\theta_k$ is taken to be the unique number such that
$\cos(\theta_k)$ is equal to the maximum of $\abs{\langle x, y \rangle}$ with unit vectors $x \in\mathcal V \cap \{ x_1,\ldots , x_{k-1} \}^\perp$ and $y \in \mathcal W \cap \{ y_1,\ldots , y_{k-1} \}^\perp$.
}
\end{definition}

Following from this definition, Bjorck and Golub \cite{bjorck1973numerical} showed that the canonical angles can be characterized in terms of the singular values of the product of two matrices that encode their respective subspace.

\begin{theorem}\cite{bjorck1973numerical} \label{bjorck}
Let $\mathcal V$ and $\mathcal W$ be subspaces of a Hilbert Space $\mathcal{H}$ with dimensions $m$, $n$ and $d$ respectively. Let $Q_{\mathcal V}$ and $Q_{\mathcal W}$ be respectively $d \times m$ and $d \times n$ matrices whose column vectors are the elements of orthonormal bases of $\mathcal V$ and $\mathcal W$ respectively represented in any orthonormal basis for $\mathcal H$.  Then the cosines of the canonical angles $\theta_k$ between the subspaces are the singular values of the $m \times n$ matrix $Q_{\mathcal V}^* Q_{\mathcal W}$, symbolically denoted by: $$cos(\theta_k) = \sigma^{\downarrow}_k (Q_{\mathcal V}^{*} Q_{\mathcal W}),$$ for all $k = 1, ..., l = \min\{m, n\}$, where $\sigma^{\downarrow}_k$ denotes the $k$th singular values of the matrix $Q^*_{\mathcal{V}} Q_{\mathcal{W}}$ listed in decreasing order. 
\end{theorem}

We can view the matrix $Q_{\mathcal V}$ in operator theoretic terms as well.  If $\mathcal{V}$ is an $m$-dimensional subspace of $\mathcal{H}$, then $Q_{\mathcal V}$ is an isometry from $\mathbb{C}^m$ into $\mathcal H$ with range equal to $\mathcal V$.  A consequence of this is that $Q_{\mathcal{V}} Q_{\mathcal{V}}^*$ is a matrix representation of the orthogonal projection from $\mathcal H$ onto $\mathcal V$, whereas on the other hand $Q_{\mathcal V}^* Q_{\mathcal V} = I_m$.

\begin{definition}
{\rm
Let $\mathcal V$ and $\mathcal W$ be two $m$-dimensional subspaces of a Hilbert space $\mathcal{H}$, where $1\le m\le \dim(\mathcal{H})$. Then $\mathcal V$ and $\mathcal W$ are said to be {\it isoclinic} if all $m$ canonical angles between $\mathcal V$ and $\mathcal W$ are equal. If that angle is $\theta$, then the subspaces are said to be {\it isoclinic at angle $\theta$}.  A collection of $m$-dimensional subspaces of a Hilbert space are said to be isoclinic if all pairs of distinct subspaces from the collection are pairwise isoclinic.
}
\end{definition}

Of course, any family of mutually orthogonal subspaces are isoclinic at angle $\frac{\pi}{2}$, but there are other possibilities as well. There are a variety of useful equivalent characterizations of isoclinic subspaces, as shown in the following result.

\begin{theorem}\label{isoprop}
Let $\mathcal V$ and $\mathcal W$ be two $m$-dimensional subspaces of a Hilbert space $\mathcal{H}$, with $m \geq 1$ and $d = \dim \mathcal H$. Let $P_{\mathcal{V}}$ and $P_{\mathcal{W}}$ denote the orthogonal projections onto the subspaces $\mathcal{V}$ and $\mathcal{W}$ respectively. Let $Q_{\mathcal{V}}$ and $Q_{\mathcal{W}}$ be $d \times m$ matrices whose column vectors are elements of the orthonormal bases of the subspaces $\mathcal{V}$ and $\mathcal{W}$ respectively, represented in any orthonormal basis for $\mathcal H$. Then the following conditions are equivalent:
\begin{itemize}
\item[$(i)$] $\mathcal V$ and $\mathcal W$ are isoclinic subspaces.
\item[$(ii)$]  $Q_{\mathcal{V}}^* Q_{\mathcal{W}}$ is a scalar multiple of a unitary on $\mathbb{C}^m$.
\item[$(iii)$] There exists $\lambda \geq 0$ such that
\begin{equation}\label{isoconds}
P_{\mathcal V}P_{\mathcal W}P_{\mathcal V} = \lambda P_{\mathcal V}   \quad \mathrm{and} \quad  P_{\mathcal W}P_{\mathcal V}P_{\mathcal W} = \lambda P_{\mathcal W} .
\end{equation}
Here, $\lambda = \cos(\theta)$ where $\mathcal V$, $\mathcal W$ are isoclinic at angle $\theta$.
\item[$(iv)$] The angle between any non-zero vector in $\mathcal{V}$ and its projection on $\mathcal{W}$ is constant; in other words,
$\Vert P_{\mathcal W} x\Vert \Vert x\Vert^{-1}$ is constant for $0\neq x \in \mathcal{V}$. And the same holds true with the roles of $\mathcal V, \mathcal W$ reversed.
\end{itemize}
\end{theorem}

\begin{proof}
The equivalence of $(i)$ and $(ii)$ follows from Theorem \ref{bjorck} above, as all the singular values of a unitary matrix are equal to one.

For $(ii) \implies (iii)$, assume $Q_{\mathcal{V}}^* Q_{\mathcal{W}}$ is a multiple of a unitary on $\mathbb{C}^m$. Then the same is true of $Q_{\mathcal W}^* Q_{\mathcal V} = (Q_{\mathcal{V}}^* Q_{\mathcal{W}})^*$.  Recall from the discussion just after Theorem~\ref{bjorck}, the projections onto subspace $\mathcal{V}$ and $\mathcal{W}$ respectively have matrix representations $P_{\mathcal{V}} = Q_{\mathcal{V}} Q_{\mathcal{V}}^*$ and $P_{\mathcal{W}} = Q_{\mathcal{W}} Q_{\mathcal{W}}^*.$
Since $Q_{\mathcal{V}}^* Q_{\mathcal{W}}$ is a multiple of a unitary on $\mathbb{C}^m$, we have for some $0\leq \lambda \leq 1$,  $Q_{\mathcal{V}}^* Q_{\mathcal{W}} Q_{\mathcal{W}}^* Q_{\mathcal{V}} = \lambda I_m$. Hence,
$$P_{\mathcal{V}} P_{\mathcal{W}} P_{\mathcal{V}} = Q_{\mathcal{V}} Q_{\mathcal{V}}^* Q_{\mathcal{W}} Q_{\mathcal{W}}^* Q_{\mathcal{V}} Q_{\mathcal{V}}^*
= Q_{\mathcal{V}} (\lambda I) Q_{\mathcal{V}}^*
= \lambda Q_{\mathcal{V}} Q_{\mathcal{V}}^*
= \lambda P_{\mathcal{V}}.$$
This is similarly done for $P_{\mathcal W}P_{\mathcal V}P_{\mathcal W} = \lambda P_{\mathcal W}.$ (Note that the $\lambda$ obtained for $Q_{\mathcal W}^* Q_{\mathcal V}$ is the same as that for $Q_{\mathcal V}^* Q_{\mathcal W} = (Q_{\mathcal W}^* Q_{\mathcal V})^*$.)

For $(iii) \implies (ii)$, assume there exists a scalar $\lambda$ such that $P_{\mathcal{V}} P_{\mathcal{W}} P_{\mathcal{V}} = \lambda P_{\mathcal{V}}$ and $P_{\mathcal{W}} P_{\mathcal{V}} P_{\mathcal{W}} = \lambda P_{\mathcal{W}}$ (necessarily $0\leq \lambda \leq 1$). Recall  $Q_{\mathcal{V}}^* Q_{\mathcal{V}} = I_m = Q_{\mathcal{W}}^*Q_{\mathcal{W}}$. Together this implies that:
\begin{align*}
P_{\mathcal{V}} P_{\mathcal{W}} P_{\mathcal{V}} =& \lambda P_{\mathcal{V}} \\
Q_{\mathcal{V}} Q_{\mathcal{V}}^* Q_{\mathcal{W}} Q_{\mathcal{W}}^* Q_{\mathcal{V}} Q_{\mathcal{V}}^* =& \lambda Q_{\mathcal{V}} Q_{\mathcal{V}}^* \\
(Q_{\mathcal{V}}^*) Q_{\mathcal{V}} Q_{\mathcal{V}}^* Q_{\mathcal{W}} Q_{\mathcal{W}}^* Q_{\mathcal{V}} Q_{\mathcal{V}}^* (Q_{\mathcal{V}}) =& (Q_{\mathcal{V}}^*) \lambda Q_{\mathcal{V}} Q_{\mathcal{V}}^* (Q_{\mathcal{V}})\\
(I) Q_{\mathcal{V}}^* Q_{\mathcal{W}} Q_{\mathcal{W}}^* Q_{\mathcal{V}} (I) =& \lambda (I) (I) \\
(Q_{\mathcal{W}}^* Q_{\mathcal{V}})^* Q_{\mathcal{W}}^* Q_{\mathcal{V}} =& \lambda I .
\end{align*}
Thus, $Q_{\mathcal{W}}^* Q_{\mathcal{V}}$ is a multiple of a unitary on $\mathbb{C}^m$. This is similarly true for $P_{\mathcal{W}} P_{\mathcal{V}} P_{\mathcal{W}} = \lambda P_{\mathcal{W}}$ and $Q_{\mathcal V}^* Q_{\mathcal W}$.

To see $(iii) \implies (iv)$, assume there exists $0\leq \lambda \leq 1$ such that $P_{\mathcal{V}} P_{\mathcal{W}} P_{\mathcal{V}} = \lambda P_{\mathcal{V}}$ and $P_{\mathcal{W}} P_{\mathcal{V}} P_{\mathcal{W}} = \lambda P_{\mathcal{W}}$. Let $0 \neq x = P_{\mathcal V} x \in \mathcal V$. Then as $P_{\mathcal{V}} P_{\mathcal{W}} P_{\mathcal{V}} = \lambda P_{\mathcal{V}}$, we have,
\[
\lambda \Vert x\Vert^2 = \lambda \langle P_{\mathcal{V}} x, x \rangle
= \langle P_{\mathcal{V}} P_{\mathcal{W}} P_{\mathcal{V}} x, x \rangle
=  \langle P_{\mathcal{W}} x, x \rangle
= \Vert P_{\mathcal W} x  \Vert^2 .
\]
Thus, $\sqrt{\lambda} = \Vert P_{\mathcal W} x\Vert \Vert x\Vert^{-1}$ for all $0 \neq x \in \mathcal V$. Similarly, from $P_{\mathcal{W}} P_{\mathcal{V}} P_{\mathcal{W}} = \lambda P_{\mathcal{W}}$, we obtain $\sqrt{\lambda} = \Vert P_{\mathcal V} x\Vert \Vert x\Vert^{-1}$ for all $0 \neq x \in \mathcal W$.

Finally for $(iv) \implies (iii)$, if $r = \Vert P_{\mathcal W} x\Vert \Vert x\Vert^{-1}$ for all $0 \neq x \in \mathcal V$, then one can follow a similar argument to that above to show $r^2 P_{\mathcal V} =  P_{\mathcal{V}} P_{\mathcal{W}} P_{\mathcal{V}}$.
\end{proof}

\begin{remark}
{\rm
We note that condition $(iv)$ was taken as the definition of isoclinic subspaces in \cite{hoggar1977new,wong1977linear}, with the equivalence of $(iii)$ and $(iv)$ being noted without proof in \cite{hoggar1977new}. The connection with canonical angles given by the equivalence of $(ii)$ and $(iii)$ appears to be new.
}
\end{remark}

\section{Connection with Quantum Error Correction}

Error models in quantum information are described by sets of operators $\{ E_i \}$ on a Hilbert space $\mathcal H$ associated with a given quantum system. In general the operators satisfy the condition $\sum_i E_i^* E_i \leq I$, which ensures the completely positive map (called a quantum channel in this context) given by $\mathcal E(\rho) = \sum_i E_i \rho E_i^*$ is a trace non-increasing map. Quantum codes are identified with subspaces $\mathcal C$ of $\mathcal H$, and the code is {\it correctable for $\mathcal E$} if there is another quantum channel $\mathcal R$ on $\mathcal H$ such that $(\mathcal R \circ \mathcal E)(\rho) = \rho$ for all density operators $\rho$ supported on $\mathcal C$.

The theory of quantum error correction grew out of seminal examples and key early results \cite{shor1995pw,steane1996error,gottesman1996d,bennett1996ch,knill2000theory}; in particular, the famous Knill-Laflamme theorem \cite{knill1997knill} is a bedrock of quantum error correction. It frames correctability of a code strictly in terms of properties of the error operators restricted to the code subspace as follows: $\mathcal C$ is correctable for $\mathcal E$ if and only if there exist scalars $\alpha_{ij}\in \mathbb{C}$ such that for all $i,j$,
\begin{equation}\label{qecconds}
P_{\mathcal C} E_i^* E_j P_{\mathcal C} = \alpha_{ij} P_{\mathcal C} ,
\end{equation}
where $P_{\mathcal C}$ is the projection of $\mathcal H$ onto $\mathcal C$. Observe that the scalars $\alpha = ( \alpha_{ij} )$ form a positive matrix.

We establish a correspondence between isoclinic subspaces and quantum error correcting codes in the following result.  Without loss of generality we will assume the code is non-degenerate in the sense that the set of restricted error operators $\{ E_i|_{\mathcal C} \}$ is minimal in size. Also, recall that an operator $U$ on a Hilbert space is a partial isometry if $U^*U$ and $UU^*$ are orthogonal projections, respectively called its initial and final projections.

\begin{theorem}\label{isoqec}
Suppose $\mathcal C$ is a subspace of a Hilbert space $\mathcal H$ that is correctable for a non-degenerate error model $\{E_i\}$. For each $i$, let $\mathcal V_i =\mathrm{Range}\, (E_i|_{\mathcal C} )$ be the range subspace of the restriction of $E_i$ to $\mathcal C$. Then $\{\mathcal V_i \}$ is a set of isoclinic subspaces of $\mathcal{H}$.
\end{theorem}

\begin{proof}
We have Eqs.~(\ref{qecconds}) satisfied for the $E_i$ and $P_{\mathcal C}$. Let $U_i$ be the partial isometries obtained through the polar decompositions of the operators $E_i P_{\mathcal C}$:
\[
E_i P_{\mathcal C} = U_i |E_i P_{\mathcal C}| = U_i \sqrt{P_{\mathcal C} E_i^* E_i P_{\mathcal C}} = \sqrt{\alpha_{ii}} U_i P_{\mathcal C}.
\]
Note that each $\alpha_{ii}\neq 0$ by non-degeneracy. We can thus reformulate the error correction conditions in terms of the $U_i$ as follows:
\[
P_{\mathcal C} U_i^* U_j P_{\mathcal C} = \frac{1}{\sqrt{\alpha_{ii}}} (P_{\mathcal C} E_i^*)  \frac{1}{\sqrt{\alpha_{jj}}} (E_j P_{\mathcal C})  = \Bigg( \frac{\alpha_{ij}}{\sqrt{\alpha_{ii}\alpha_{jj}}} \Bigg) P_{\mathcal C}.
\]
Also observe that for each $i$, by construction we have $P_i := U_i P_{\mathcal C} U_i^*$ is the projection onto the range $\mathcal V_i$ of $E_i P_{\mathcal C}$ and $P_{\mathcal C}=  P_{\mathcal C} U_i^* U_i P_{\mathcal C}$.

Now for each pair $i,j$, let $\lambda_{ij} = \alpha_{ij}(\sqrt{\alpha_{ii}\alpha_{jj}})^{-1}$ and note that $\overline{\lambda_{ij}} = \lambda_{ji}$. Then we have:
\begin{eqnarray*}
P_i P_j P_i &=&  P_i U_j (P_{\mathcal C} U_j^* U_i P_{\mathcal C}) U_i^* \\
&=& \lambda_{ji} P_i U_j P_{\mathcal C} U_i^* \\
&=&   \lambda_{ji}   U_i (P_{\mathcal C} U_i^* U_j P_{\mathcal C}) U_i^* \\
&=&   \lambda_{ji} \lambda_{ij}  U_i P_{\mathcal C} U_i^* \\
&=&   |\lambda_{ij}|^2   P_i .
\end{eqnarray*}
Similarly, $P_j P_i P_j = |\lambda_{ij}|^2   P_j$. As $\mathcal V_i = P_i\mathcal H$, it follows from Theorem~\ref{isoprop} that the subspaces $\{ \mathcal V_i \}$ are isoclinic.
\end{proof}

We present the following example of a simple error model to illustrate this result.

\begin{example}\label{qeceg}
{\rm
Consider a two-qubit error model describing a bit flip on the first qubit with the probability of some fixed $0 < p < 1$. We can formulate this mathematically by taking $\ket{ij} = \ket{i} \otimes \ket{j}$, $i,j = 0,1$, as a fixed orthonormal basis for $\mathbb{C}^4 = \mathbb{C}^2 \otimes \mathbb{C}^2$. Then if we let $X$ be the Pauli bit flip operator ($X\ket{0}=\ket{1}$, $X\ket{1}=\ket{0}$), we can define $X_1 = X \otimes I_2$ and the error model as a map on two-qubit density operators is given by:
\[
\mathcal E(\rho) = (1-p) \rho + p \, X_1 \rho X_1^*.
\]
Here the error operators are $E_1 = \sqrt{1-p}\, I_4$ and $E_2 = \sqrt{p} X_1$.

Now define two subspaces of $\mathbb{C}^4$ as follows: $\mathcal C_1 = \mathrm{span}\{ \ket{00},\ket{11} \}$ and  $\mathcal C_2 = \mathrm{span}\{ \ket{10},\ket{01} \}$. Let $P_1$, $P_2$ be the corresponding projections. Then $\mathcal C_1$ (and similarly $\mathcal C_2$) is a correctable code for $\mathcal E$, with $\mathcal C_1$, $\mathcal C_2$ the relevant family of subspaces as in the theorem, and in this case the matrix $\alpha = (\alpha_{ij})$ satisfies $\alpha_{11} = 1-p$, $\alpha_{22}= p$, $\alpha_{12} = \alpha_{21} = 0$. So here the canonical angles are both equal to $\theta = \frac{\pi}{2}$ (indeed we have $P_1 P_2 = 0 = P_2 P_1$), and the subspaces are isoclinic.

We can complicate things slightly and obtain more interesting isoclinic subspace structure. Suppose the system is exposed to noise that induces a rotation of angle $0 < \phi < 2\pi$ to the original error model; that is, the original error operators are replaced by
\[
F_1 = (\cos\phi ) E_1 + (\sin\phi) E_2 \quad \mathrm{and} \quad F_2 = (-\sin\phi) E_1 + (\cos\phi) E_2,
\]
which can also be seen through the matrix relation $[F_1 \, F_2] = [E_1 \, E_2] U$ where $U$ is the rotation matrix $U = \left( \begin{array}{cc} \cos\phi & -\sin\phi \\ \sin\phi & \cos\phi  \end{array}  \right).$

The Knill-Laflamme conditions show that correctable codes are the same for error models whose operators are linear combinations of each other, hence $\mathcal C_1$ is correctable for $\{ F_1, F_2\}$. Indeed, here we have, with $c = \cos \phi$, $s = \sin \phi$,
\begin{eqnarray*}
P_{\mathcal C} F_1^* F_1 P_{\mathcal C} &=&  ( c^2 (1-p) + s^2 p ) P_{\mathcal C}   \\
P_{\mathcal C} F_2^* F_2 P_{\mathcal C} &=&  ( s^2 (1-p) + c^2 p ) P_{\mathcal C}
\end{eqnarray*}
and
\[
P_{\mathcal C} F_1^* F_2 P_{\mathcal C} = \big( cs(2p-1) \big) P_{\mathcal C}  = P_{\mathcal C} F_2^* F_1 P_{\mathcal C}.
\]
One can check that the unitary $U$ factors through to give the new error correction coefficient matrix as $\alpha' = U^* \alpha U$. Moreover, the isoclinic angle $\theta$ is computed from the proof of Theorem~\ref{isoqec} in terms of the rotation $\phi$ and probability $p$ as follows:
\[
\theta = \cos^{-1} \Big(   \frac{|cs(2p-1)|^2}{(c^2(1-p) + s^2p)(s^2(1-p) + c^2p)} \Big) .
\]
See the figure below for a 3-space depiction of $\theta\in [0,\frac{\pi}{2}]$ as it depends on $0\leq p \leq 1$ and $0\leq \phi \leq 2 \pi$.}

\begin{figure}[htbp]
    \label{isograph}
    \centering
    \hspace{0cm}
    \includegraphics[scale=0.33]{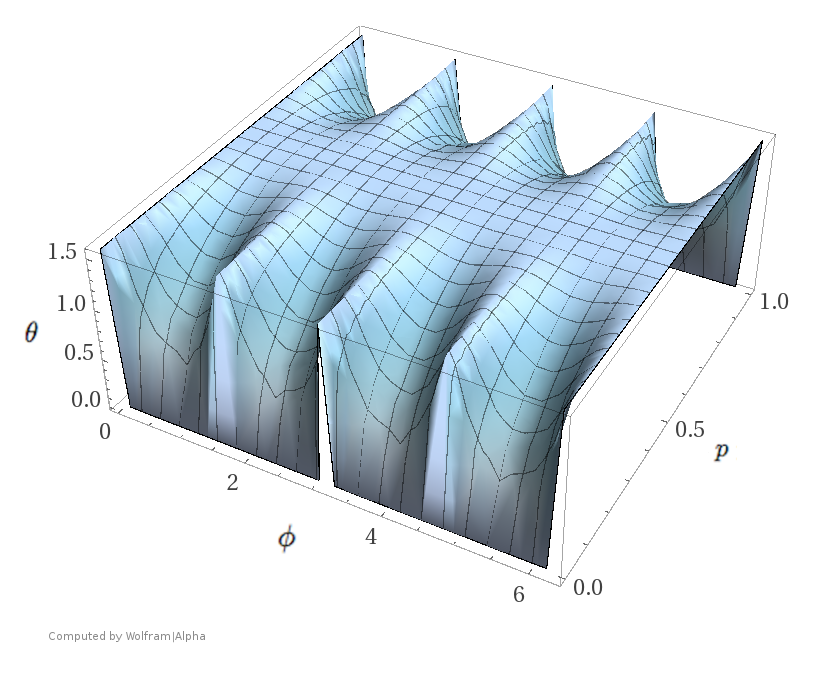}
    \caption{The dependence of $\theta$ on $p$, $\phi$ derived in Example~\ref{qeceg}.}
    \label{fig:lines}
\end{figure}

\end{example}


There is at least a partial converse of the above theorem given as follows.

\begin{proposition}\label{isoqecconverse}
Let $\mathcal{H}$ be a Hilbert space and let $\{ P_1, P_2 \}$ be a pair of projections on $\mathcal H$ associated with two $m$-dimensional isoclinic subspaces. Then each of the subspaces $P_i \mathcal H$ is correctable for the error model $\{ \frac{1}{\sqrt{2}} P_i \}_{i=1}^2$.
\end{proposition}

\begin{proof}
The projections $P_1$, $P_2$ satisfy the isoclinic identities Eq.~(\ref{isoconds}), with say $P_i P_j P_i = \lambda P_i$ for $i\neq j \in \{ 1,2\}$. Hence we have
\begin{eqnarray*}
P_1 P_1^* P_2 P_1 = P_1 P_2 P_1 = \lambda P_1.
\end{eqnarray*}
Similar identities hold for each product $P_i P_j^* P_k P_i$, $i,j,k = 1,2$, and the result follows from the quantum error correction conditions of Eq.~(\ref{qecconds}).
\end{proof}

Motivated by this result, we finish this section by presenting an example of a pair of isoclinic subspaces that arise in matrix theory and Euclidean geometry, found in Wong's original monograph \cite{wong1977linear}.

\begin{example}
{\rm
Given a $2 \times 2$ complex matrix $M$, one can consider the {\it graph of $M$} which is the subspace of $\mathbb{C}^4$ given by:
\[
\mathcal V_{M} := \Big\{ \begin{pmatrix} x \\ Mx \end{pmatrix} : x \in \mathbb{C}^2 \Big\}.
\]
The orthogonal complement of $\mathcal V_M$ inside $\mathbb{C}^4$ is given as follows:
\[
\mathcal V_{M}^\perp := \Big\{ \begin{pmatrix} -M^* x \\ x \end{pmatrix} : x \in \mathbb{C}^2 \Big\}.
\]
By direct calculation one can show the orthogonal projection of $\mathbb{C}^4$ onto $\mathcal V_M$ is given in block matrix form as (writing $I$ for $I_2$):
\[
P_M : =
    \begin{pmatrix} I \\ M \end{pmatrix} (I + M^* M )^{-1} \begin{pmatrix} I & M^* \end{pmatrix}.
\]

Isoclinic subspaces can be obtained in this context via solutions to certain matrix equations. As in \cite{wong1977linear}, one can solve for $2 \times 2$ matrices $A$ and $B$ and scalar $\lambda$ such that
\[
    \begin{pmatrix} I & A^* \end{pmatrix} \begin{pmatrix} I \\ B \end{pmatrix} (I + B^* B)^{-1} \begin{pmatrix} I & B^* \end{pmatrix}  \begin{pmatrix} I \\ A \end{pmatrix} (I + A^* A)^{-1} = \lambda ;
\]
in other words,
\[
    (I + A^* B) (I + B^* B)^{-1}  (I + B^* A) = \lambda  (I + A^* A) .
\]
With this equation satisfied, we can use the decomposition of $P_A$ and $P_B$ derived above in the general case to conclude that $P_A P_B P_A = \lambda P_A$. A pair of matrices that satisfies this equation, with $A$ and $B$ in either role and $\lambda = \frac12$, is given by:
\[
A = \begin{bmatrix} 1 & 0 \\ 0 & -1\end{bmatrix} \quad \mathrm{and}  \quad B = \begin{bmatrix} \frac{\sqrt{3} + 1}{\sqrt{3} - 1} & 0 \\ 0 & 0\end{bmatrix}.
\]

Thus, it follows from Theorem~\ref{isoprop} that $\mathcal V_A$ and $\mathcal V_B$ are isoclinic at angle $\theta = \frac{\pi}{3}$. Moreover, the quantum error correction conditions can be verified directly in this case as in the proof of Proposition~\ref{isoqecconverse}.
}
\end{example}

\section{Higher Rank Numerical Ranges and Isoclinic Subspaces}

We can also derive a connection with the higher rank numerical range of a matrix or operator. Originally considered in the setting of quantum error correction \cite{choi2005quantum,choi2006higher2}, these numerical ranges have been intensely investigated for over a decade now in matrix theory and beyond \cite{li2007higher,woerdeman2008higher,martinez2008higher,choi2008geometry,li2008canonical,li2009condition,li2011generalized,GLPS}.

Given an operator or matrix $A$ on $\mathbb{C}^n$ and $1\leq k \leq n$, the {\it rank-$k$ numerical range of $A$} is the subset of the complex plane given by:
\[
\Lambda_k(A) = \big\{ \lambda \in \mathbb{C} \, \big| \, PAP = \lambda P \, \mathrm{for\, some \, rank-}k\,\mathrm{projection}\,P\,\mathrm{on}\,\mathbb{C}^n   \big\}.
\]

Here we are interested in the case of higher rank numerical ranges of projections, which can be viewed as a special case of Hermitian operators considered in \cite{choi2006higher2}. If $P$ is a non-zero projection with $\mathrm{rank}(P) = l < n$, then through an application of Theorem~2.4 from \cite{choi2006higher2}, it follows that $\Lambda_k(P) = [0,1]$ whenever $k \leq \min\{ l, n-l \}$.

\begin{proposition}
Let $P$ and $Q$ be nonzero projections on $\mathbb{C}^n$ of the same rank $1 \leq k \leq n$. Then $P \mathbb{C}^n$ and $Q \mathbb{C}^n$ are isoclinic subspaces at angle $\theta$ if and only if $P Q P = \cos (\theta) P$ if and only if $Q P Q = \cos (\theta) Q$.
\end{proposition}

\begin{proof}
Firstly, the case that $\theta =  \frac{\pi}{2}$ and $\cos(\theta) = 0$ corresponds to orthogonality of the two subspaces and $PQ=0=QP$. So let us assume $\cos(\theta)\neq 0$ for the rest of the proof.

Suppose $P Q P = \cos (\theta) P$, and so
$$
(QPQ)(QPQ) = QP(QQ)PQ = Q(PQP)Q = \cos (\theta) QPQ.
$$
Next, dividing both sides by $\cos^{2} (\theta)$ we get,
$$
\frac{1}{\cos^2 (\theta)} (QPQ)(QPQ) = \frac{1}{\cos (\theta)} QPQ.
$$
Hence $\cos^{-1}(\theta) QPQ$ is a projection that is evidently supported on $Q \mathbb{C}^n$. However, we also have, with $\Tr(\cdot)$ the trace functional,
\begin{eqnarray*}
\Tr(\frac{1}{\cos (\theta)} QPQ) = \frac{1}{\cos (\theta)} \Tr (QPQ)
&=& \frac{1}{\cos (\theta)} \Tr (QP)  \\
&=& \frac{1}{\cos (\theta)} \Tr(PQP)
= \Tr(P) = \Tr(Q).
\end{eqnarray*}
As the rank of a projection is equal to its trace, it follows that in fact $QPQ = \cos(\theta) Q$.

Thus we have shown that $PQP = \cos(\theta) P$ if and only if $QPQ = \cos(\theta) Q$. The equivalence of these conditions with $P\mathbb{C}^n$ and $Q\mathbb{C}^n$ being isoclinic follows from Theorem~\ref{isoprop}.
\end{proof}

\begin{remark}
{\rm
In particular, for the projections $P$, $Q$ corresponding to a pair of isoclinic subspaces, each of the projections is encoded into the structure of the other projection's higher rank numerical ranges in the sense that: $P$ (respectively $Q$) is a projection corresponding to $\cos (\theta) \in \lambda_k (Q)$ (respectively $\Lambda_k (P)$).
}
\end{remark}

\vspace{0.1in}

{\noindent}{\it Acknowledgements.} D.W.K. was partly supported by NSERC and a University Research Chair at Guelph. R.P. was partly supported by NSERC.

\bibliography{bib-2}
\bibliographystyle{amsplain}

\end{document}